\documentclass[11pt]{article}

\usepackage{a4wide}
\usepackage{macros}
\usepackage{amssymb}

\title{Simple proof of the completeness theorem \\ for \\ 
second order classical and intuitionistic logic \\ by \\ reduction to 
first-order mono-sorted logic\footnote{Paru dans : Theoretical
  Computer Science, vol 308, pp. 227-237, 2003}}
\author{Karim NOUR \\Christophe RAFFALLI\\
\small LAMA - Equipe de Logique\\
\small Universit\'e de Chamb\'ery\\
\small 73376 Le Bourget du Lac\\
\small e-mail nour@univ-savoie.fr, raffalli@univ-savoie.fr}
\date{}

\begin{document}

\maketitle

\begin{abstract}
We present a simpler way than usual to deduce the completeness theorem
for the second-oder classical logic from the first-order one. We also
extend our method to the case of second-order intuitionistic logic.
\end{abstract}

\section{Introduction}
The usual way (but not the original Henkin's proof \cite{Hen49,Hen50}) 
for proving the completeness theorem for second-order logic is to
deduce it from the completeness theorem for first-order multi-sorted
logic \cite{Fef74}. There is clearly a trivial translation from second-order
logic to first-order multi-sorted
logic, by associating one sort to first-order objects and, for each $n
\in {\mathbb N}$, one sort for predicates of arity $n$.

Another way (due Van Dalen \cite{VDal94}) to is to deduce it from the completeness theorem for
first-order mono-sorted logic: Van Dalen method's is to associate a first-order variable $x$ to each second-order
variable $X$ of ariry $n$, and encode the atomic formula $X(x_1,\dots,x_n)$
by $\Ap_n(x,x_1,\dots,x_n)$ where $\Ap_n$ is a relation
symbol of arity $n+1$. Then, this coding is extended to all
formulas.We write it $F \mapsto F^*$. However, to allow the translation
between second-order proofs and first-order proofs, one adds some
axioms to discriminate between first and second-order objects.
The critical point is the translation of quantifications:

\begin{itemize}
\item For first-order quantification we define 
$(\forall x\, F)^* = \forall x (v(x) \rightarrow F^*)$ where $v$ is a new
predicate constant.
\item For second-order quantification of arity $n$ we define 
$(\forall X^n\, F)^* = \forall x (V_n(x) \rightarrow  F^*)$ where $V_n$ is a new
predicate constant.
\end{itemize}
Then we add axioms relating $v$, $V_n$ and $\Ap_n$ such as 
$\forall x\forall y (\Ap_1(x,y) \rightarrow V_1(x) \et v(y))$.

The problem is that this translation is not surjective. So it is
not immediate to prove that if $F^*$ is provable in first-order logic
then $F$ is provable in second-order logic, because all the formulas
appearing in the proof of $F^*$ are not necessarily of the shape
$G^*$. It is not even clear that the proof in \cite{VDal94} which is
only sketched can be completed into a correct proof (at least the
authors do not know how to end his proof). May be there is a solution
using the fact that subformulas of $F^*$ are nearly of the shape $G^*$
and one could use this in a direct, but very tedious, proof by
induction on the proof of $F$ using the subformula property which is a
strong result.

Our solution, is to simplify Van Dalen's translation $F \mapsto F^*$
from second-order logic to first-order. The novelty of this paper is
to replace Van Dalen's axiom's and extra predicate constant by a
coding $F \mapsto \rev{F}$ from first-order logic to second-order such
that $\rev{F^*}$ and $F$ are logically equivalent. To achieve this we
consider that in first order logic the same variable may have
different meanings (in the semantics) depending on it's position in
atomic formulas. Thus, we can translate any first-order formula back
to a second-order formula.

Using this method we can also deduce a definition of Kripke models
\cite{Kri65} for second-order intuitionistic logic and easily get a
completeness theorem. This models are similar to Prawitz's
second-order Beth's models \cite{Pra70,Bet56}.

This was not at all so clear with Van Dalen's method (as we do not how
to end his proof) if we need classical absurdity to use the extra
axioms. We also give some simple examples showing that despite a
complex definition, computation is possible in these models.

\medskip
{\bf Acknowledgement.}  We wish to thank both referees for their
comments that helped a lot to improve the paper and Miss Christelle
Favre for her assistance in the checking for certain proofs while she
was preparing her master-thesis.

\section{Coding}

\begin{definition}[second-order language]
Let $\Lang_2$, the language of second-order logic, be the following:
\begin{itemize}
\item The logical symbols $\faux$,$\fl$ , $\et$, $\ou$, $\q$ and $\e$.
\item A countable set ${\cal V}$ of first-order variables :
$x_0, x_1, x_2, \dots$
\item A countable set $\Sigma$ of constants and functions symbols (of
various arity) : $a, b, f, g, h, \dots$.
\item Using ${\cal V}$ and $\Sigma$ we construct the set of
first-order terms $\cal T$ : $t_1,t_2,...$
\item For each $n \in \mathbb{N}$, a countable set ${\cal V}_n$ of
second-order variables of arity $n$ : $X_0^n , X_1^n , X_2^n, \dots$.
\end{itemize}
\end{definition}

To simplify, we omit second-order constants (they can be replaced by
free variables).

 
\begin{definition}[first-order language]
Let $\Lang_1$, a particular language of first-order logic, be the following:
\begin{itemize}
\item The logical symbols $\faux$,$\fl$ , $\et$, $\ou$, $\q$ and $\e$.
\item A countable set ${\cal V}$ of first-order variables :
$x_0, x_1, x_2,\dots$ (it is simpler to use the same set of first-order variables in $\Lang_1$ and $\Lang_2$).
\item A countable set $\Sigma$ of  constants and functions symbols (of
various arity) : $a, b, f, g, h, \dots$. Here again we use the same
set as for $\Lang_2$.
\item For each $n \in \mathbb{N}$, a relation symbol $\Ap_n$ of arity $n+1$.
\end{itemize}
\end{definition}

\paragraph{Notations}
\begin{itemize}
\item We write $\Free(F)$ for the set of all free variables of a
formula $F$.
\item We write $F \equi G$ for $(F \fl G) \et (G \fl F)$.
\item We write $F[x:=t]$ for the first-order substitution of a term.
\item We write $F[X^n:=Y^n]$ for the second-order substitution of a variable.
\item We write $F[X^n:=\lambda x_1\dots x_n G]$ for the second-order
substitution of a formula.
\item We will use natural deduction \cite{Pra65,VDal94} both for second and first-order
logic, and we will write $\Gamma \vdash_k^n F$ with $k \in \{i , c \}$
(for intuitionistic or classical logic) and $n \in \{1 , 2 \}$ (for
first or second-order).
\end{itemize}

We have the following lemma:

\begin{lemma}\label{substitution}
If  $\Gamma \vdash_k^n A$ then, for every substitution
$\sigma$,  $\Gamma[\sigma] \vdash_k^n A[\sigma]$.
\end{lemma}

\begin{definition}[coding]
We choose for each $n \in \mathbb{N}$ a bijection $\phi_n$ from
${\cal V}_n$ to ${\cal V}$. The fact that it is a bijection for
each $n$ is the main point in our method.

Let $F$ be a second-order formula, we define a first-order formula $F^*$ by induction as
follows:
\begin{itemize}
\item $\faux^* = \faux$
\item $(X^n(t_1,\dots,t_n))^* = \Ap_n(\phi_n(X^n),t_1,\dots,t_n)$
\item $(A \Diamond B)^* = A^*\Diamond  B^*$ where $\Diamond \in \{ \fl ,
\et, \ou \}$
\item $(Q x\, A)^* = Q y (A[x:=y])^*$ where $y \not \in \Free(A^*)$ and
 $Q \in \{ \q, \e \}$
\item $(Q X^n\, A)^* = Q y (A[X^n:=Y^n])^*$ where $\Phi_n(Y^n) = y$, $y
\not \in \Free(A^*)$ and $Q \in \{ \q, \e \}$
\end{itemize}
\end{definition}

\begin{remark} In the coding, the same free first-order variable (this
will not be the case for bound ones) has different meanings depending on
its location in the translated formula. 
\end{remark}

\begin{example} $(\q X (X(x) \fl X(y)))^* = \q z (\Ap_1(z,x) \fl
\Ap_1(z,y))$.  This example illustrates why we need renaming. For
instance, if $\Phi_1(X)$ were equal to $x$ or $y$ in $(X(x) \fl X(y))^*$.
\end{example}

\begin{remark} The mapping $F \mapsto F^*$ is not surjective, for
instance there is no antecedent for $\q x\,\Ap_1(x,x)$ or $\Ap_1(f(a),a)$.
\end{remark}

\begin{definition}[comprehension schemas]
The second-order comprehension schema $SC_2$ is the set of all closed formulas 
$SC_2(G;x_1,\dots,x_n;\chi_1,\dots,\chi_m)$ where $\{x_1,\dots,x_n\}
\subset {\cal V}$ and $\Free(G) \subseteq \{x_1,\dots,x_n,\chi_1,\dots,\chi_m\}$ and 
$$SC_2(G;x_1,\dots,x_n;\chi_1,\dots,\chi_m) = \q \chi_1\dots\q \chi_m \exists X^n \q x_1\dots\q x_n \left(G \equi
X^n(x_1,\dots,x_n)\right) \in SC_2$$
where $X^n \not \in F_v(G)$.

The first-order comprehension schema $SC_1$ is defined simply as
$SC_2^* = \{F^*, F \in SC_2\}$
\end{definition}

It is easy to show that $SC_2$ is provable in second order logic. 

\begin{remark}\label{remarkSC}
Let $F = X(x)$ where $\Phi_1(X) = x$. We have:
\begin{itemize}
\item $SC_2(F;x;X) = \q X \e Y \q x (F \equi Y(x)) \in SC_2$.
\item $SC_2(F;x;X)^* = (\q X \e Y \q x (F \equi Y(x)))^* = \q z \e y \q x
(Ap_1(z,x) \equi Ap_1(y,x)) \in SC_1$.
\end{itemize}
It is easy to see that $(\q X \e Y \q x (F \equi Y(x)))^* = 
\q z \e y \q x (F[X := Z]^* \equi Ap_1(y,x))$ where $\phi_1(Z) = z
\not = x$.

In general we have the following result : for each second-order formula $G$
there is a variable substitution $\sigma$ such that
$$
\begin{array}{rcl}
SC_2(G;x_1,\dots,x_n;\chi_1,\dots,\chi_m)^* &=&
\left(\q \chi_1\dots\q \chi_m \exists X^n \q x_1\dots\q x_n \left(G \equi
X^n(x_1,\dots,x_n)\right)\right)^* \cr &=&
\q y_1\dots\q y_m \exists x \q x_1\dots\q x_n \left(G[\sigma]^* \equi 
Ap_n(x,x_1,\dots,x_n)\right). \cr 
\end{array}
$$
\end{remark}

We can now show the following theorem (we will not use it):

\begin{theorem}\label{transproofun}
Let $\G$ be a second-order context and $A$ a second-order formula.
If $\G \vdash_{k}^{2} A$ then $\G^*,SC_1 \vdash_{k}^{1} A^*$  ($k
\in \{ i , c \}$).
\end{theorem}

\begin{proof}
By induction on the derivation of $\G \vdash_{k}^{2} A$, using $SC_1$,
remark \ref{remarkSC} and lemma \ref{substitution}
for the case of the second-order elimination of $\q$ and the
second-order introduction of $\e$. \qed
\end{proof}

\begin{definition}[reverse coding]
Let $F$ be a first-order formula, we define a second-order formula
$\rev{F}$ by induction as follows:
\begin{itemize}
\item $\rev{\faux} = \faux$
\item $\rev{Ap_n(x,t_1,\dots,t_n)} = X^n(t_1,\dots,t_n)$ where $X^n= \phi^{-1}_n(x)$
\item $\rev{Ap_n(t,t_1,\dots,t_n)} = \faux$ if $t$ is not a variable.
\item $\rev{(A \Diamond B)} = \rev{A} \Diamond  \rev{B}$ where $\Diamond \in \{ \fl ,
\et, \ou \}$
\item $\rev{(Q x\,A)} = Q x Q X^{i_1} \dots Q X^{i_p} \rev{A}$ where 
$Q \in \{ \q, \e \}$, $X^n= \phi^{-1}_n(x)$ for all $n \in \mathbb{N}$,
$i_1 < i_2 < \dots < i_p$ and $\{X^{i_1},\dots,X^{i_p}\}  =
{\cal V}_n \cap  \Free(\rev{A})$
\end{itemize}
\end{definition}

\begin{remark}
We don't need renaming in order to define $\rev{(Q x\,A)}$ since the $\phi_n$ are bijections.
\end{remark}

\begin{lemma}\label{idempotent}
If $A$ is a second order formula then $\vdash_{i}^{2} \rev{A^*} \equi A$.
\end{lemma}

\begin{proof}
By induction on the formula $A$. \qed
\end{proof}

\begin{remark}
The embarrassing case of decoding $Ap_n(t,t_1,\dots,t_n)$ (where $t$
is not a variable) never arises here since we only decode encoded
formulas. We can not say that $\rev{A^*} = A$, because in the case of
the quantifier, we can add or remove some quantifiers on variables
with no occurrence. For instance, if $X^0 \neq Y^0$, $\Phi_0(X^0) =
x$ and $\Phi_0(Y^0) =
y$ then  $\rev{(\q X^0\,Y^0)^*} = \rev{(\q x\,Ap_0(y))} = \q x\,Y^0$. \qed
\end{remark}

\begin{corollary}\label{SCdeux}
$\vdash_{i}^{2} \rev{(SC_1)} \equi SC_2$ which means that each formula
in $\rev{(SC_1)}$ is equivalent to at least one formula in $SC_2$ and
vice versa.
\end{corollary}

\begin{proof}
Consequence of \ref{idempotent}. \qed
\end{proof}

\begin{example}\label{exempleutile} The aim of this example is to give
an idea of the proof of lemma \ref{transproofdeux}.

Let  $\G$ be a first-order context, $F = Ap_1(x,y) \fl Ap_2(x,y,y) \ou Ap_1(y,x)$ and $t$ a term.

We have : 
\begin{itemize}
\item $\rev{(\q x \,F)} = \q x \q X^1 \q X^2 (X^1(y) \fl X^2(y,y) \ou
Y^1(x))$ and $\rev{(\e x \,F)} = \e x \e X^1 \e X^2 (X^1(y) \fl X^2(y,y) \ou
Y^1(x))$ (where $\phi_1(Y^1) = y$).
\item If $t = z$, then $\rev{(F[x:=t])} =  Z^1(y) \fl Z^2(y,y) \ou Y^1(z)$
(where $\phi_1(Z^1) = \phi_2(Z^2) = z$) and if $t$ is not a variable, then $\rev{(F[x:=t])} =  \perp \fl \perp \ou Y^1(t)$
\end{itemize}
We remark that :
\begin{itemize}
\item $\rev{(F[x:=z])} = Z^1(y) \fl Z^2(y,y) \ou
Y^1(z) = \rev{F}[X^1 := Z^1][x := z]$ if $z$ is a variable such that $\phi_1(Z^1) = \phi_2(Z^2) = z$.
\item $\rev{(F[x:=t])} = \perp
\fl \perp \ou Y^1(t) = \rev{F}[X^1 := \lambda x_1 \perp][x := t]$ if $t$ is not a variable.
\end{itemize}
and then :
\begin{itemize}
\item If $\rev{\G} \vdash_{k}^{2} \rev{(\q x \,F)}$, then (by using
some $\q$-elimination rules) $\rev{\G} \vdash_{k}^{2} \rev{(F[x:=t])}$.
\item If $\rev{\G} \vdash_{k}^{2} \rev{(F[x:=t])}$, then (by using
some $\e$-introduction rules) $\rev{\G} \vdash_{k}^{2} \rev{(\e x \,F)}$.
\end{itemize}
\end{example}

\begin{lemma}\label{transproofdeux}
Let $\G$ be a first-order context and $A$ a first-order formula.
If $\G\vdash_{k}^{1} A$ then $\rev{\G} \vdash_{k}^{2} \rev{A}$  ($k
\in \{ i , c \}$).
\end{lemma}

\begin{proof} By induction on the derivation of $\G \vdash_k^1 A$.
The only difficult cases are the case of the elimination of $\q$ and the
 introduction of $\e$ which are treated in the same way as the
 examples \ref{exempleutile}. \qed
\end{proof}
\medskip
Now, we can prove the converse of theorem
\ref{transproofun}, which is the main tool to prove our completeness theorems:

\begin{theorem}\label{transprooftrois}
Let $\G$ be a second-order context and $A$ a second-order formula.
If $\G^*,SC_1 \vdash_{k}^{1} A^*$ then $\G \vdash_{k}^{2} A$  ($k
\in \{ i , c \}$).
\end{theorem}

\begin{proof} By lemma \ref{transproofdeux}, corollary \ref{SCdeux}, lemma
\ref{idempotent} and using the fact that formulas in $SC_2$ are provable.\qed
\end{proof}

\section{Classical completeness}

Here is the usual definition of second order models \cite{Man96,Pra67,VDal94}:

\begin{definition}[second-order classical model]
A second-order model for $\Lang_2$ is given by a tuple 
${\cal M}_2 = ({\cal D} , \overline{\Sigma}, 
\{{\cal P}_n\}_{n \in \mathbb{N}})$ where
\begin{itemize}
\item ${\cal D}$ is a non empty set.
\item $\overline{\Sigma}$ contains a function $\overline{f}$ from 
${\cal D}^n$ to ${\cal D}$ for each function $f$ of arity $n$ in $\Sigma$.
\item ${\cal P}_n \subseteq {\cal P}({\cal D}^n)$ for each $n \in \mathbb{N}$.
The set  ${\cal P}_n$ of subsets of ${\cal D}^n$ will be used as the range
for the second-order quantification of arity $n$.
For $n=0$, we assume that ${\cal P}_0 = {\cal P}({\cal D}^0) =
\{0,1\}$  because ${\cal P}({\cal D}^0) = {\cal P}(\emptyset) =
\{\emptyset, \{\emptyset\}\} = \{0,1\}$.
\end{itemize}

An ${\cal M}_2$-interpretation $\sigma$ is a function on $\Var
\cup \bigcup_{n \in \mathbb{N}} \Var_n$ such that $\sigma(x) \in {\cal D}$ for $x
\in \Var$ and $\sigma(X^n) \in {\cal P}_n$ for $X^n \in \Var_n$.

If $\sigma$ is a ${\cal M}_2$-interpretation, we define
$\sigma(t)$ the interpretation of a first-order term by induction with 
$\sigma(f(t_1,\dots,t_n)) = \overline{f}(\sigma(t_1),\dots,\sigma(t_n))$.


Then if $\sigma$ is a ${\cal M}_2$-interpretation we
define ${\cal M}_2, \sigma \models A$ for a formula $A$ by
induction as follows:
\begin{itemize}
\item ${\cal M}_2, \sigma  \models X^n(t_1,\dots,t_n)$ iff $(\sigma(t_1),\dots,\sigma(t_n)) \in \sigma(X^n)$

\item ${\cal M}_2, \sigma \models A \rightarrow B$ iff 
${\cal M}_2, \sigma \models A$ implies ${\cal M}_2, \sigma \models B$

\item ${\cal M}_2, \sigma \models A \wedge B$ iff 
${\cal M}_2, \sigma \models A$ and ${\cal M}_2, \sigma \models B$

\item ${\cal M}_2, \sigma \models A \vee B$ iff 
${\cal M}_2, \sigma \models A$ or ${\cal M}_2, \sigma \models B$

\item ${\cal M}_2, \sigma \models \forall x\,A$ iff 
for all $v \in {\cal D}$ we have ${\cal M}_2, \sigma[x := v] \models A$

\item ${\cal M}_2, \sigma \models \exists x\,A$ iff there exists 
$v \in {\cal D}$ such that ${\cal M}_2, \sigma[x := v] \models A$

\item ${\cal M}_2, \sigma \models \forall X^n\,A$ iff 
for all $\pi \in {\cal P}_n$ we have ${\cal M}_2, \sigma[X^n := \pi] \models A$

\item ${\cal M}_2, \sigma \models \exists X^n\,A$ iff there exists 
$\pi \in {\cal P}_n$ such that ${\cal M}_2, \sigma[X^n := \pi] \models A$
\end{itemize}

We will write ${\cal M}_2 \models A$ if for all ${\cal
M}_2$-interpretation $\sigma$ we have ${\cal M}_2, \sigma
\models A$.
\end{definition}
  
\begin{definition}[first-order classical model]
A first-order model for $\Lang_1$ is given by a tuple 
${\cal M}_1 = ({\cal D} , \overline{\Sigma}, \{\alpha_n\}_{n \in \mathbb{N}})$ where
\begin{itemize}
\item ${\cal D}$ is a non empty set.
\item $\overline{\Sigma}$ contains a function $\overline{f}$ from 
${\cal D}^n$ to ${\cal D}$ for each function $f$ of arity $n$ in $\Sigma$.
\item $\alpha_n \subseteq {\cal D}^{n+1}$ for each $n \in \mathbb{N}$.
The relation $\alpha_n$ will be the interpretation of $\Ap_n$.
\end{itemize}

An ${\cal M}_1$-interpretation $\sigma$ is a function from $\Var$ to 
${\cal D}$.

For any first-order model ${\cal M}_1$, any first-oder
formula $A$ and any ${\cal M}_1$-interpretation $\sigma$, we define 
${\cal M}_1, \sigma \models A$ et ${\cal M}_1 \models A$
 as above by induction on $A$
(we just have to remove the cases for second-order quantification). 
\end{definition}

\begin{definition}[semantical translation]
Let ${\cal M}_1 = ({\cal D} , \overline{\Sigma}, \{\alpha_n\}_{n \in \mathbb{N}})$
 be a first-order  model. We define a second-order model
$\rev{{\cal M}_1} = ({\cal D} ,  \overline{\Sigma} , \{{\cal P}_n\}_{n \in \mathbb{N}})$
 where ${\cal P}_0 = \{0,1\}$
and for $n > 0$, ${\cal P}_n = \{|a|_n; a \in {\cal D}\}$ where $|a|_n
= \{(a_1, \dots, a_n) \in {\cal D}^n;
 (a,a_1, \dots, a_n) \in \alpha_n\}$.

Let $\sigma$ be an ${\cal M}_1$-interpretation, we define
$\rev{\sigma}$ an $\rev{{\cal M}_1}-interpretation$ by
$\rev{\sigma}(x) = \sigma(x)$ if $x \in \Var$
and
$\rev{\sigma}(X^n) = |\sigma(\phi(X^n))|_n$.
\end{definition}

\begin{lemma}\label{csemone} 
For any first-order  model ${\cal M}_1$, any ${\cal
M}_1$-interpretation $\sigma$ and any second order formula $A$,
${\cal M}_1, \sigma \models A^*$ if and only if
$\rev{{\cal M}_1}, \rev{\sigma} \models A$.
\end{lemma}

\begin{proof}
By induction on the formula $A$, this is an immediate consequence of
the definition of semantical translation. \qed
\end{proof}

\begin{corollary}\label{csemtwo} 
For any first-order  model ${\cal M}_1$, 
${\cal M}_1 \models  SC_1$ if and only if 
$\rev{{\cal M}_1} \models SC_2$.
\end{corollary} 

\begin{proof}
Immediate consequence of lemma \ref{csemone} using the fact that formulas in
$SC_1$ and $SC_2$ are closed. \qed
\end{proof}

\begin{theorem}[Completeness of second order classical semantic]\label{ccomplet}
Let A be a closed second-order formula. $\vdash_{c}^{2} A$ iff
for any second-order model ${\cal M}_2$ such that 
${\cal M}_2 \models  SC_2$ we have ${\cal M}_2 \models A$. 
\end{theorem}

\begin{proof}
$\Longrightarrow$ Usual direct proof by induction on the proof of
$\vdash_{c}^{2} A$.

$\Longleftarrow$ Let ${\cal M}_1$ be a first-order model such that
${\cal M}_1 \models  SC_1$.
Using corollary \ref{csemtwo}  we have 
$\rev{{\cal M}_1} \models  SC_2$ and by hypothesis, we get
 $\rev{{\cal M}_1} \models A$. Then using lemma \ref{csemone}
we have ${\cal M}_1 \models A^*$. As this is true for any 
first-order model satisfying $SC_1$, the first-order completeness
theorem gives $SC_1 \vdash_{c}^{1} A^*$ and this leads to the wanted result 
$\vdash_{c}^{2} A$ using theorem \ref{transprooftrois}. \qed
\end{proof}

\section{Intuitionistic completeness}

Our method, when applied to the intuitionistic case, gives the
following definition of second-order models (similar to Prawitz's
adaptation of  Beth's
models \cite{Pra70}). We mean that the definition arises mechanically
if we want to get lemma \ref{isemone} (which is the
analogous of lemma \ref{csemone} in the classical case).

\begin{definition}[second-order intuitionistic model]
A second-order Kripke model for $\Lang_2$ is given by a tuple 
${\cal K}_2 = ({\cal K}, 0, \leq, \{{\cal D}_p\}_{p \in {\cal K}} , 
\{\overline{\Sigma}_p\}_{p \in {\cal K}},
\{\Pi_{n,p}\}_{n \in \mathbb{N}, p \in {\cal K}})$ where
\begin{itemize}
\item $({\cal K}, \leq, 0)$ is a partially ordered set with $0$ as bottom element.
\item ${\cal D}_p$ are non empty sets such that 
for all $p,q \in {\cal K}$, $p \leq q$ implies ${\cal D}_p \subseteq
{\cal D}_q$.
\item $\overline{\Sigma}_p$ contains a function $\overline{f}_p$ from 
${\cal D}^n_p$ to ${\cal D}_p$ for each function $f$ of arity $n$ in
$\Sigma$. Moreover, for all $p,q \in {\cal K}$, $p \leq q$ implies
that for all $(a_1,\dots,a_n) \in {\cal D}^n_p \subseteq {\cal D}^n_q$
we have $\overline{f}_p(a_1,\dots,a_n) = \overline{f}_q(a_1,\dots,a_n)$.
\item $\Pi_{n,p}$ are non empty sets of increasing functions $
(P_q)_{q \geq p}$ such that for all $q \geq p, P_q \in
{\cal P}({\cal D}^n_q)$ (increasing means for all $q,q' \geq p$, $q \leq q'$
implies $P_q \subseteq P_{q'}$). Moreover, if $q \geq p$ and $\pi \in
\Pi_{n,p}$ then $\pi$ restricted to all $q' \geq q$ belongs to $\Pi_{n,q}$.

In particular, an element of
$\Pi_{0,p}$ is a particular increasing function in \{0,1\} with $0 = \emptyset$ and
$1 = \{\emptyset\}$.
\end{itemize}

A ${\cal K}_2$-interpretation $\sigma$ at level $p$ is a function
$\sigma$ such that 
$\sigma(x) \in {\cal D}_p$ for $x
\in \Var$ and $\sigma(X^n) \in \Pi_{n,p}$ for $X^n \in \Var_n$. 
\end{definition}

\begin{remark} If $\sigma$ is a ${\cal K}_2$-interpretation
at level $p$ and $p \leq q$ then $\sigma$ can be considered as ${\cal
K}_2$-interpretation at level $q$ by restricting all the values of
second order variables to $q' \geq q$. Then we write ${\cal K}_2,
\sigma, q \realise A$ even if $\sigma$ is defined at a level $p \leq
q$. This is used mainly in the definition of the interpretation of
implication.
\end{remark}

\begin{definition}
If $\sigma$ is a ${\cal K}_2$-interpretation at level p, we
define $\sigma(t)$ the interpretation of a first-order term by induction with 
$\sigma(f(t_1,\dots,t_n)) = \overline{f}_p(\sigma(t_1),\dots,\sigma(t_n))$.

Then if $\sigma$ is a ${\cal K}_2$-interpretation at level $p$ we
define ${\cal K}_2, \sigma, p \realise A$ for a formula $A$ by
induction as follows:
\begin{itemize}
\item ${\cal K}_2, \sigma, p \realise X^n(t_1,\dots,t_n)$ iff $(\sigma(t_1),\dots,\sigma(t_n)) \in \sigma(X^n)(p)$

\item ${\cal K}_2, \sigma, p \realise A \rightarrow B$ iff for all $q \geq p$ if
${\cal K}_2, \sigma, q \realise A$ then ${\cal K}_2, \sigma, q \realise B$

\item ${\cal K}_2, \sigma, p \realise A \wedge B$ iff 
${\cal K}_2, \sigma, p \realise A$ and ${\cal K}_2, \sigma, p \realise B$

\item ${\cal K}_2, \sigma, p \realise A \vee B$ iff 
${\cal K}_2, \sigma, p \realise A$ or ${\cal K}_2, \sigma, p \realise B$

\item ${\cal K}_2, \sigma, p \realise \forall x\,A$ iff for all $q \geq
p$, for all $v \in {\cal D}_q$ we have ${\cal K}_2, \sigma[x := v], q \realise A$

\item ${\cal K}_2, \sigma, p \realise \exists x\,A$ iff there exists 
$v \in {\cal D}_p$ such that ${\cal K}_2, \sigma[x := v], p \realise A$

\item ${\cal K}_2, \sigma, p \realise \forall X^n\,A$ iff for all $q \geq
p$, for all $\pi \in \Pi_{n,q}$ we have ${\cal K}_2, \sigma[X^n := \pi], q \realise A$

\item ${\cal K}_2, \sigma, p \realise \exists X^n\,A$ iff there exists 
$\pi \in \Pi_{n,p}$ such that ${\cal K}_2, \sigma[X^n := \pi], p \realise A$
\end{itemize}

We will write ${\cal K}_2 \realise A$ if for all ${\cal
K}_2$-interpretation $\sigma$ at level $0$ we have ${\cal K}_2, \sigma, 0 \realise A$.
\end{definition}

\begin{remark}
 Interpretations are monotonic, this means that the set of true
statements only increase when we go from world $p$ to world $q$ with
$p \leq q$.
\end{remark}
\medskip

We recall here the usual Kripke's definition \cite{Kri65} of intuitionistic models:

\begin{definition}[first-order intuitionistic model]
A first-order Kripke model is given by a tuple 
${\cal K}_1 = ({\cal K}, 0, \leq, \{{\cal D}_p\}_{p \in {\cal K}} , 
\{\overline{\Sigma}_p\}_{p \in {\cal K}},
\{\alpha_{n,p}\}_{n \in \mathbb{N}, p \in {\cal K}}, \realise)$ where
\begin{itemize}
\item $({\cal K}, \leq, 0)$ is a partially ordered set with $0$ as bottom element.
\item ${\cal D}_p$ are non empty sets such that 
for all $p,q \in {\cal K}$, $p \leq q$ implies ${\cal D}_p \subseteq
{\cal D}_q$.
\item $\overline{\Sigma}_p$ contains a function $\overline{f}_p$ from 
${\cal D}^n_p$ to ${\cal D}_p$ for each function $f$ of arity $n$ in
$\Sigma$. Moreover, for all $p,q \in {\cal K}$, $p \leq q$ implies
that for all $(a_1,\dots,a_n) \in {\cal D}^n_p \subseteq {\cal D}^n_q$
we have $\overline{f}_p(a_1,\dots,a_n) = \overline{f}_q(a_1,\dots,a_n)$.
\item $\alpha_{n,p}$  are subsets of  ${\cal D}_p^{n+1}$ such that 
for all $p,q \in {\cal K}$, for all $n \in \mathbb{N}$, 
$p \leq q$ implies $\alpha_{n,p}\subseteq
\alpha_{n,q}$.
\item $\realise$ is the relation defined by 
$p \realise \Ap_n(a,a_1,\dots,a_n)$ if and only if 
$p \in {\cal K}$ and
$(a,a_1,\dots,a_n) \in \alpha_{n,p}$.
\end{itemize}

A ${\cal K}_1$-interpretation $\sigma$ at level $p$ is a function from $\Var$ 
to ${\cal D}_p$.

For any first-order Kripke model ${\cal K}_1$, any first-oder formula
$A$ and any ${\cal K}_1$-interpretation $\sigma$, we define ${\cal
K}_1, p, \sigma \realise A$ as above.

We will write ${\cal K}_1 \realise A$ iff for ${\cal
K}_1$-interpretation $\sigma$ at level $0$ we have
${\cal K}_1, \sigma, 0 \realise A$.
\end{definition}

\begin{definition}[semantical translation]
Let 
$${\cal K}_1 = ({\cal K}, 0, \leq, \{{\cal D}_p\}_{p \in {\cal K}} , 
\{\overline{\Sigma}_p\}_{p \in {\cal K}},
\{\alpha_{n,p}\}_{n \in \mathbb{N}, p \in {\cal K}}, \realise)$$
 be a first-order Kripke model. We define a second-order Kripke model
$$\rev{{\cal K}_1} = ({\cal K}, 0, \leq, \{{\cal D}_p\}_{p \in {\cal K}}, 
\{\overline{\Sigma}_p\}_{p \in {\cal K}}, 
\{\Pi_{n,p}\}_{n \in \mathbb{N}, p \in {\cal K}})$$
 where  $\Pi_{n,p} = \{|a|_n; a \in {\cal D}_p\}$ with for all $q
\geq p$, $|a|_n(q) = \{(a_1, \dots, a_n) \in {\cal D}^n_q;
 (a,a_1, \dots, a_n) \in \alpha_{n,q}\}$.

Let $\sigma$ be a ${\cal K}_1$-interpretation at level $p$, we define
$\rev{\sigma}$ a $\rev{{\cal K}_1}$-interpretation at level $p$ by
$\rev{\sigma}(x) = \sigma(x)$ 
and
$\rev{\sigma}(X^n) = |\sigma(\phi(X^n))|_n$.
\end{definition}

\begin{lemma}\label{isemone} 
For any first-order Kripke model ${\cal K}_1$, any ${\cal
K}_1$-interpretation $\sigma$ at level $p$ and any second order formula $A$,
${\cal K}_1,\sigma,p \realise A^*$ if and only if
$\rev{{\cal K}_1},\rev{\sigma},p \realise A$.
\end{lemma}

\begin{proof}
By induction on the formula $A$, this is an immediate consequence of
the definition of semantical translation. \qed
\end{proof}

\begin{corollary}\label{isemtwo} 
For any first-order Kripke model ${\cal K}_1$, 
${\cal K}_1 \realise  SC_1$ if and only if 
$\rev{{\cal K}_1} \realise SC_2$.
\end{corollary} 

\begin{proof}
Immediate consequence of lemma \ref{isemone}. \qed
\end{proof}

\begin{theorem}[Completeness of second order intuitionistic semantic]\label{icomplet}
Let A be a closed second-order. $\vdash_{i}^{2} A$ iff
for all second-order Kripke model ${\cal K}_2$ such that 
${\cal K}_2 \realise  SC_2$ we have ${\cal K}_2 \realise A$. 
\end{theorem}

\begin{proof}
$\Longrightarrow$ Usual direct proof by induction on the proof of
$\vdash_{i}^{2} A$.

$\Longleftarrow$ Identical to the proof of theorem \ref{ccomplet}
using the lemmas \ref{isemone} and \ref{isemtwo} instead of lemmas
\ref{csemone} and \ref{csemtwo}.\qed
\end{proof}

\section{Examples of second order propositional intuitionistic models}

In this section we will only consider propositional intuitionistic
logic. Then the definition of models can be simplified using the
following remark:

\begin{remark}
The interpretation of a propositional variable at level $p$ can be
seen as a bar: a bar being a set $\cal B$ with 
\begin{itemize}
\item for all $q \in \cal B$,  $q \geq p$    
\item for all $q, q' \in \cal B$ such that $q \neq q'$, we have neither $q \leq q'$ nor  $q' \leq q$
\end{itemize}

In the case of finite model, there is a canonical isomorphism between
the set of bars and the set of increasing functions in $\{0, 1\}$ by
associating to a bar $\cal B$ the function $\pi$ such that $\pi(q) =
1$ if and only if there exists $r \in \cal B$ such that $q \geq r$.
This usually helps to ``see'' the interpretation of a formula.

This is not the case for infinite model, if we consider $\mathbb Q^+$,
the set of rational greater than $\sqrt{2}$ is not a bar. 
\end{remark}

\begin{example}
We will now construct a counter model for the universally quantified
Peirce's law: $P = \forall X \forall Y (((X \rightarrow Y) \rightarrow
X) \rightarrow X)$: We take a model ${\cal K}_2$ with two points $0,p$
and such that $\Pi_{0,0}$ contains $\pi_1$ and $\pi_2$ defined by
$\pi_1(0) = \pi_2(0) = \pi_2(p) = 0$ and $\pi_1(p) = 1$ (this means
that $\pi_2$ is the empty bar and $\pi_1$ is then bar $\{p\}$). It is clear
that ${\cal K}_2, \sigma[X:=\pi_1,Y:=\pi_2], 0 \not\realise ((X
\rightarrow Y) \rightarrow X) \rightarrow X$. So we have ${\cal K}_2
\not\realise P$. We can also remark that this model is not full the
bar $\{0\}$ is missing.

\end{example}

A natural question arises: if one codes as usual conjunction,
disjunction and existential using implication and second order
universal quantification what semantics is induced by this coding?
If we keep the original conjunction, disjunction and existential, it
is obvious that the defined connective are provably equivalent to the
original ones, and therefore, have the same semantics.

However, if we remove conjunction, disjunction and existential from
the model we only have the following:

\begin{proposition}
The semantics induced by the second order coding of conjunction,
disjunction and existential is the standard Kripke's semantics if the
model is full (that is if $\Pi_{n,p}$ is the set of all increasing
functions with the desired properties).
\end{proposition}
 
\begin{proof}
\begin{description}
\item[$A \wedge B = \forall X ((A \rightarrow (B \rightarrow X)) \rightarrow X)$]: We must prove
that ${\cal K}_2, \sigma, p \realise A \wedge B$ if and only if 
${\cal K}_2, \sigma, p \realise A$ and ${\cal K}_2, \sigma, p \realise
B$. The right to left implication is trivial. For the left to right,
we assume ${\cal K}_2, \sigma, p \realise A \wedge B$, We consider the
interpretation $\pi$ defined by $\pi(q) = 1$ if and only if 
${\cal K}_2, \sigma, q \realise A$ and ${\cal K}_2, \sigma, q \realise
B$. Then it is immediate that 
${\cal K}_2, \sigma[X := \pi], p \realise A \rightarrow (B \rightarrow X)$. So we have 
${\cal K}_2, \sigma[X := \pi], p \realise X$ which means that 
$\pi(p) = 1$ which is equivalent to ${\cal K}_2, \sigma, p \realise A$
and ${\cal K}_2, \sigma, p \realise
B$.

\item[$A \vee B = \forall X ((A \rightarrow X) \rightarrow (B \rightarrow X) \rightarrow X)$]: The proof is
similar using $\pi$ defined by 
$\pi(q) = 1$ if and only if 
${\cal K}_2, \sigma, q \realise A$ or ${\cal K}_2, \sigma, q \realise
B$.
\item[$\exists \chi\,A = \forall X (\forall \chi (A \rightarrow X) \rightarrow X)$]: The proof is
similar using $\pi$ defined by 
$\pi(q) = 1$ if and only if there exists $\phi$
a possible interpretation for $\chi$ such that
${\cal K}_2, \sigma[\chi := \phi], q \realise A$.\qed

\end{description}
\end{proof}

\begin{remark}
If we compare this proof to the proof in \cite{Kri90e,Par88} about data-types
in AF2, we remark that second order intuitionistic models are very
similar to realizability models. Moreover, in both cases, we are in general
unable to compute the semantics of a formula if the model is not full
(for realizability, not full means that the interpretation of second
order quantification is an intersection over a strict subset of the
set of all sets of lambda-terms).

Moreover, the standard interpretation of the conjunction is
${\cal K}, \sigma, p \realise A \wedge B$ if and only if ${\cal K},
\sigma, p \realise A$ and ${\cal K}, \sigma, p \realise B$. However,
if the model is not full and if the language does not contain the
conjunction, the function $\pi$ defined for $q \geq p$ by $\pi(p) = 1$
if and only if ${\cal K}, \sigma, q \realise A$ and ${\cal K},
\sigma, q \realise B$ does not always belong to $\Pi_{0,p}$. In this
case, the interpretation of the second order definition of the
conjunction is strictly smaller than the natural interpretation.

It would be interesting to be able to construct such non standard
model, but this is very hard (due to the comprehension schemas). In
fact the authors do not know any practical way to construct such a non
full model. In the framework of realizability, such non full model
would be very useful to prove that some terms are not typable of type
$A$ in Girard's system F while they belong to the interpretation of
$A$ in all full models (for instance Maurey's term for the $\inf$
function on natural number).
\end{remark}

\begin{small}
\baselineskip=0.17in
\nocite{*}
\bibliography{biblio.bib}
\bibliographystyle{plain}
\end{small}

\end{document}